\documentclass{amsart}[12pt]
\usepackage[]{amsmath, amsthm, amsfonts, verbatim, amssymb}
\usepackage{tikz}
\usetikzlibrary{matrix,arrows}
\usepackage{url}
\usepackage[all,cmtip]{xy}

\usepackage{xcolor}

\usepackage{graphicx}
\usepackage{pinlabel}

\def\cH{\mathcal H}
\def\cO{{\mathcal O}}

\def\bZ{{\mathbb Z}}

\def\bC{{\mathbb C}}

\begin{document}
\newtheorem {theo}{Theorem}
\newtheorem {coro}{Corollary}
\newtheorem {lemm}{Lemma}
\newtheorem {rem}{Remark}
\newtheorem {defi}{Definition}
\newtheorem {ques}{Question}
\newtheorem {prop}{Proposition}
\newtheorem {conj}{Conjecture}
\def\spb{\smallpagebreak}
\def\mpb{\vskip 0.5truecm}
\def\bpb{\vskip 1truecm}
\def\wtM{\widetilde M}
\def\tM{\widetilde M}
\def\wtN{\widetilde N}
\def\tN{\widetilde N}
\def\tR{\widetilde R}
\def\tC{\widetilde C}
\def\tX{\widetilde X}
\def\tY{\widetilde Y}
\def\tE{\widetilde E}
\def\tH{\widetilde H}
\def\tL{\widetilde L}
\def\tQ{\widetilde Q}
\def\tS{\widetilde S}
\def\tc{\widetilde c}
\def\talpha{\widetilde\alpha}
\def\ti{\widetilde \iota}
\def\hM{\hat M}
\def\hq{\hat q}
\def\hR{\hat R}
\def\bs{\bigskip}
\def\ms{\medskip}
\def\ni{\noindent}
\def\td{\nabla}
\def\pd{\partial}
\def\hol{$\text{hol}\,$}
\def\Log{\mbox{Log}}
\def\bfQ{{\bf Q}}
\def\Todd{\mbox{Todd}}
\def\top{\mbox{top}}
\def\Pic{\mbox{Pic}}
\def\bP{{\mathbb P}}
\def\dxi{d x^i}
\def\dxj{d x^j}
\def\dyi{d y^i}
\def\dyj{d y^j}
\def\dzi{d z^I}
\def\dzj{d z^J}
\def\ozi{d{\overline z}^I}
\def\ozj{d{\overline z}^J}
\def\oz1{d{\overline z}^1}
\def\oz2{d{\overline z}^2}
\def\oz3{d{\overline z}^3}
\def\sI{\sqrt{-1}}
\def\hol{$\text{hol}\,$}
\def\ok{\overline k}
\def\ol{\overline l}
\def\oJ{\overline J}
\def\oT{\overline T}
\def\oS{\overline S}
\def\oV{\overline V}
\def\oW{\overline W}
\def\oY{\overline Y}
\def\oL{\overline L}
\def\oI{\overline I}
\def\oK{\overline K}
\def\oL{\overline L}
\def\oj{\overline j}
\def\oi{\overline i}
\def\ok{\overline k}
\def\oz{\overline z}
\def\om{\overline mu}
\def\on{\overline nu}
\def\oa{\overline \alpha}
\def\ob{\overline \beta}
\def\oGamma{\overline \Gamma}
\def\of{\overline f}
\def\oN{\overline N}
\def\og{\overline \gamma}
\def\ogamma{\overline \gamma}
\def\odelta{\overline \delta}
\def\otheta{\overline \theta}
\def\ophi{\overline \phi}
\def\opd{\overline \partial}
\def\oA{\overline A} 
\def\oB{\overline B}
\def\oC{\overline C}
\def\oD{\overline D}
\def\oIq1{\oI_1\cdots\oI_{q-1}}
\def\oIq2{\oI_1\cdots\oI_{q-2}}
\def\op{\overline \partial}
\def\ua{{\underline {a}}}
\def\us{{\underline {\sigma}}}
\def\Chow{{\mbox{Chow}}}
\def\vol{{\mbox{vol}}}
\def\dim{{\mbox{dim}}}
\def\rank{{\mbox{rank}}}
\def\diag{{\mbox{diag}}}
\def\tor{\mbox{tor}}
\def\supp{\mbox supp}
\def\bp{{\bf p}}
\def\bk{{\bf k}}
\def\a{{\alpha}}
\def\tchi{\widetilde{\chi}}
\def\ta{\widetilde{\alpha}}
\def\ovarphi{\overline \varphi}
\def\ocH{\overline{\cH}}
\def\tV{\widetilde{V}}
\def\tf{\widetilde{f}}
\def\th{\widetilde{h}}
\def\tT{\widetilde T}
\def\hG{\widehat{G}}
\def\hS{\widehat{S}}
\def\hD{\widehat{D}}
\def\Aut{\mbox{Aut}}
\def\hX{\widehat{X}}
\def\hC{\widehat{C}}
\def\hs{\widehat{s}}

\ni
\title[Syzygies of ample line bundles on fake projective planes]{On the syzygies of ample line bundles \\ on fake projective planes}

\author[]
{Sui-Chung Ng and Sai-Kee Yeung}

\begin{abstract}  
{\it Our goal is to study the syzygies of the projective embeddings defined by the ample line bundles on a fake projective plane $S$. The syzygies are studied in terms of the property $N_p$.  For various kinds of ample line bundles, we give explicit lower bounds for their powers above which the property $N_p$ is satisfied.}
\end{abstract}

\address[]{S.-C. Ng, School of Mathematical Sciences, Shanghai Key Laboratory of PMMP, East China Normal University, Minhang District, Shanghai, China}
\address[]{S.-K. Yeung, Mathematics Department, Purdue University, West Lafayette, IN  47907 USA}
\email{scng@math.ecnu.edu.cn}
\email{yeung@math.purdue.edu}

\thanks{
\noindent{The first author was partially supported by National Natural Science Foundation of China (Grant No. 11501205) and Science and Technology Commission of Shanghai Municipality (Grant No. 13dz2260400).
The second author was partially supported by grants from the National Science Foundation, USA}}

\ni{\it }
\maketitle

%\begin{center}
%{\bf 1. Introduction} 
%\end{center} 

\section{Introduction}

Studying the very-ampleness, projective normality and syzygies of the ample line bundles on a complex surface is natural and has a long history.  The main purpose
of this paper is to give some detailed estimates for fake projective planes.  

 There are two natural questions for a projective algebraic manifold $M$ of complex dimension $n$
equipped with an ample line bundle $L$.  The first is to study the smallest multiple of $L$ that gives
a projective embedding of the manifold, and the second is to study the syzygy relative to such an embedding.  For the first question, the Fujita Conjecture states that $K_M\otimes L^k$
is very ample if $k\geqslant n+2$.  For the second question, there is the Mukai Conjecture which states that $K_M\otimes L^{n+2+p}$ satisfies the $N_p$ property.
For surfaces, Fujita Conjecture is proved by Reider [R].  The Mukai Conjecture is mostly open.

Fake projective planes are interesting surfaces
 to study among smooth surfaces of general type since they have the smallest Euler number (viz., $3$) among all such surfaces and have been classified in ~[PY1], [PY2] and ~[CS].   However, there is still no good way to describe such surfaces directly from classical algebraic geometry in general.
 For a fake projective plane $S$, $K_S^3$ is very ample from a well-known result of Bombieri and also Reider [R].   
 We also know that $K_S$ has
 no sections.  The natural question of whether $K_S^2$ is very ample or not is not settled, even though there are some examples which are true.
%(cf.~[Y2] and the references therein).  
Hence one would expect that sharp results for the syzygy of $S$ may be difficult to achieve.  
 The interest about syzygy is that it would provide
 information about the pluricanonical ring of $S$. 
 %which may eventually lead to explicit descriptions of a fake projective plane, an interesting problem which is widely open.  
 As a consequence of the current work, it turns out that we have pretty good information about the
 syzygy, cf. Corollary 1, even though it is still short of Mukai Conjecture.  The results of this paper are close to the optimal ones with the currently available 
 techniques.  The difficulty to achieve a sharper result is related to the difficulty 
 %in getting the very-ampleness of $K_S^2$, which in turn is
regarding the existence of exceptional collection of objects, cf. Conjecture 1 in \S5.  Sharper estimates are obtained for some special fake projective planes as explained in Theorem 9 and the remark following it in \S5.

%Let $S$ be a fake projective plane and $K_S$ be its canonical line bundle.
%Though there is no non-trivial global section of $K_S$ on $S$, it is known from the classical result of Bombieri that $K^3_S$ is very ample, and  
%$K^2_S$ is very ample for many of them (see~[Y2] and the references therein).

%To study further the structure of a fake projective plane, a natural problem is the syzygies of the projective embeddings 
%defined by its ample line bundles.  
 The syzygy of a given projective embedding can be described using the so-called property $N_p$, where $p$ is a non-negative integer. 
 For the details of its definition, we refer the reader to~[EL]. Here, we simply recall that the property $N_0$ corresponds to projective normality and the property $N_1$ corresponds to projective normality together with the condition that the homogeneous ideal of the embedded subvariety is generated by quadratic polynomials. 
 The present paper gives effective results for the lower bounds of the powers of a given ample line bundle  satisfying the property $N_p$ on a fake projective plane. %A fake projective plane is in particular a smooth complex 2-ball quotient and thus a complex surface of general type. 
 The reader may refer to ~[R\'em] or ~[Y] for basic
 geometric facts about fake projective planes.
 
 For the general study of syzygies related to surfaces of general type, the reader can see, for example, the works of  Banagere-Hanumanthu~[BH], Gallego-Purnaprajna~[GP], Hwang-To~[HT], Purnaprajna~[P] and the very recent work of Niu~[N]. Here we remark that unlike the previous results found in the literature for surfaces of general type, which are pertaining to base-point-free line bundles and are expressed in terms of some sort of ``regularity" related the relevant line bundles, we established in this paper explicit and concrete numbers for a given arbitrary ample line bundle, which is not necessarily base-point-free, and in particular, including the  canonical bundle (which has no non-trivial section).

Our first main result is regarding an arbitrary ample line bundle on a fake projective plane.

\begin{theo}\label{mainthm2}
Let $A$ and $N$ be line bundles on a fake projective plane $S$ such that $A$ is ample and $N$ is nef. Then, $A^{6(p+3)}\otimes N$ satisfies the property $N_p$, where $p\geq 0$.
\end{theo}

It follows easily from the definition that any fake projective plane $S$ is of Picard number 1. Moreover, it is known~(see for example~[Y]) that the canonical line bundle can be written as $K_S=G^3\otimes\tau$, where $G$ is an ample generator $G$ of Pic$(S)$ and $\tau$ is some torsion line bundle.  Thus, we have the following immediate corollary of Theorem~\ref{mainthm2}.

\begin{coro}\label{maincoro}
Let $S$ be a fake projective plane. Then $K^m_S$ satisfies the property $N_p$ if $m\geq 2p+6$, where $p\geq 0$. In particular, $K^m_S$ is projectively normal if $m\geq 6$. 
\end{coro}

\noindent\textbf{Remark.} We have also obtained results for smaller values of $m$, but those require various additional efforts. In fact, we will see that $K^{2p+2}_S$ (resp. $K^{2p+4}_S$) also satisfies the property $N_p$ for $p\geq 1$ (resp. $p\geq 0$), as a consequence of Theorem~\ref{mainthm} in Section~3. In particular, $K^4_S$ is also projectively normal. Furthermore, motivated by the existence problem for exceptional collections for objects in the derived category of coherent sheaves on a fake projective plane $S$, if we assume that for any ample generator $G$ of Pic$(S)$ and any torsion line bundle $\delta$, we have $h^0(S, G^2\otimes\delta)=0$, a vanishing statement which is a variant of a conjecture mentioned in~[GKMS], then we will be able to show that that $K^{2p+5}_S$ also satisfies the property $N_p$ for $p\geq 0$ (Section~5). Such conjecture is known to be true for fake projective planes with more than three automorphisms. As an intermediate step of the proof, we also obtained some vanishing or almost vanishing results concerning the torsion line bundles of a fake projective plane,
which are of some independent interests.
 
 \begin{prop}
 Let $\delta$ be a torsion line bundle on a fake projective plane $S$ and $G$ an ample generator of $\Pic(S)$.  Then\\
 (a). $h^1(S,\delta)=0$;\\
 (b). $h^0(S, G\otimes\delta)\leq 1$, $h^0(S, G^2\otimes\delta)\leq 2$ and $h^0(S, G^3\otimes\delta)\leq 1$.
 \end{prop}

Finally, we mention the following theorem for an arbitrary base-point-free line bundle on a fake projective plane, which will be established in Section 4.
%, using a computer-aided result provided by Donald Cartwright. 

\begin{theo}\label{mainthm3}
Let $S$ be a fake projective plane and $B$ be a base-point-free line bundle on $S$. Then, $B^m$ is projectively normal if $m\geq 2$. Moreover, for $p\geq 1$, $B^m$ satisfies the property $N_p$ if $m\geq p+1$. 
\end{theo}

 This follows from a vanishing result on the unramified cover of each fake projective plane defined by the commutator subgroup of its fundamental group.

The results in this paper are established through the well known cohomology criterion given by Green~[G]. The details of Green's criterion will be recalled in Section 2. To show that the cohomology groups appearing in Green's theorem vanish, our main tool is the classical result of Castelnuovo-Mumford on the surjectivity of certain tensor product maps on various spaces of global sections, which has been used already in the literature, in particular in the work of Gallego-Purnaprajna~[GP] for surfaces under various assumptions and for the properpty $N_p$, $p\geqslant 1$.  Our goal is to provide as sharp an estimate as possible for fake projective planes, especially for lower values of $p$.  Careful analysis is needed to make sure that there is no numerical gap in the range of possible powers of the line bundle satisfying the theorems of this paper, (see the remark before Theorem \ref{vanishingthm2}).
We have also made an
effort to make the the presentation reasonably self-contained.

\section{Cohomological criterion for the property $N_p$}

We will first recall in this section the cohomological criterion for the property $N_p$ obtained by Green~[G].

Let $X$ be a complex projective manifold and $L$ be an ample and base-point-free line bundle on $X$. Denote by $\mathcal H^0(X,L)$ the trivial vector bundle on $X$ whose fibers are $H^0(X,L)$. Then we have the following exact sequence of vector bundles 
$$
	0\rightarrow M_L\rightarrow \mathcal H^0(X,L)\rightarrow L\rightarrow 0,
$$
where the fiber of $M_L$ at $x\in X$ is the subspace in $H^0(X,L)$ consisting of the sections vanishing at $x$. Using the theory of Koszul cohomology, Green~[G] obtained the following 
\begin{theo}[Green]
The line bundle $L$ satisfies the property $N_p$ if $$H^1(X, \wedge^r M_L\otimes L^s)=0$$ for $0\leq r\leq p+1$ and $s\geq 1$.
\end{theo}

Over a field of characteristic zero, $\wedge^r M_L\subset M_L^{\otimes r}$ is a direct summand and therefore it is a common practice to simply verify the vanishing of $H^1(X, M_L^{\otimes r}\otimes L^s)$ for studying the property $N_p$. The main tool for us to achieve this is the following theorem, pertaining to the so-called \textit{Castelnuovo-Mumford regularity}~(cf. [L]).

\begin{theo}[Castelnuovo-Mumford] \label{regularity}
Let $E$ be a base-point-free line bundle on a projective variety $M$ and $\mathcal F$ be a coherent sheaf on $M$.  Suppose that $H^i(M,\mathcal F\otimes E^{-i})=0$ for $i\geqslant 1$.  Then
the multiplication mapping 
$$H^0(M,\mathcal F\otimes E^j)\otimes H^0(M,E)\rightarrow H^0(M,\mathcal F\otimes E^{j+1})$$
is surjective for all $j\geqslant 0$.
\end{theo}

\section{Vanishing theorems and proofs of the main theorems}

In what follows, we let $S$ be a fake projective plane. We recall some known facts about fake projective planes, which can be found in the survey articles by R\'emy [R\'em] and Yeung~[Y]. First of all, the Picard number of $S$ is 1 and we will fix an ample generator $G$ of $\textrm{Pic}(S)$. We have then $G\cdot G=1$. Moreover, if $K_S$ denotes the canonical line bundle, then $K_S=G^3\otimes\tau$, for some torsion line bundle $\tau$. 

\begin{lemm}\label{vanishinglem}
	Let $A$ be an ample line bundle and $\sigma$ be a torsion line bunde on $S$. Then, $H^1(S, A^k\otimes\sigma)=H^2(S,A^k\otimes\sigma)=0$ for $k\geq 4$. 
\end{lemm}
\begin{proof}
Since $A$ is ample and the Picard number of $S$ is 1, we can write $A=G^m\otimes\epsilon$ for some $m\in\mathbb N^+$ and some torsion line bundle $\epsilon$. Then, for $i=1,2$ and $k\geq 4$,
$$
	H^i(S, A^k\otimes\sigma)=H^i(S, K_S\otimes G^{mk-3}\otimes\tau^{-1}\otimes\epsilon^k\otimes\sigma)=0
$$
by Kodaira Vanishing Theorem as $G^{mk-3}$ is ample.
\end{proof}

\begin{prop}\label{vanishingprop}
Let $B$ be an ample line bundle on $S$ such that $B\otimes K_S^{-2}$ is nef, or equivalently, $B=G^m\otimes\sigma$ for some $m\geq 6$ and some torsion line bundle $\sigma$. Then $B^k$ is base-point-free for $k\geq 1$ and $H^1(S, B^k)=H^2(S, B^k)=0$ for $k\geq 0$. 
\end{prop}
\begin{proof}

By hypotheses, we can write $B^k=K_S\otimes L$, where $L=G^{m'} \otimes \sigma'$ for some $m'\geq 3$ and some torsion line bundle $\sigma'$. Since $L\cdot L =(m')^2\geq 9$, using the classical results of Reider~[Rei], we see that $B^k$ is base-point-free. For $k\geq 1$, the vanishing of $H^1(S, B^k)$ and $H^2(S, B^k)$ follows directly from Lemma~\ref{vanishinglem}, while for $k=0$, it follows from the fact that $H^1(S,\mathcal O_S)=0$ and $H^2(S,\mathcal O_S)\cong H^0(S, K_S)=0$ for any fake projective plane.
\end{proof}

In the rest of this section, we will establish two vanishing theorems for the cohomology groups appearing in Green's criterion for the property $N_p$. As mentioned before, our main tool is the theorem of Castelnuovo-Mumford (Theorem~\ref{regularity}). The basic ideas behind our methods are taken from the work of Gallego-Purnaprajna~[GP]. For base-point-free line bundles on $S$, we streamline the arguments in~[GP] and adapt to the present context for fake projective planes. Geometric properties of fake projective planes are utilized to make sure that conditions for the Castelnuovo-Mumford theorem are satisfied.

From Proposition~\ref{vanishingprop}, if $B$ is an ample line bundle such that $B\otimes K_S^{-2}$ is nef, then $B^k$ is base-point-free for $k\geq 1$. Thus, we have following exact sequence of vector bundles
$$
	0\rightarrow M_{B^k}\rightarrow\mathcal H^0(S, B^k)\rightarrow B^k\rightarrow 0.
$$

Tensoring with $M^{\otimes (r-1)}_{B^k}\otimes B^\ell$ for $r\geq 1$ and $\ell\geq 0$, we have
$$
	0\rightarrow M^{\otimes r}_{B^k}\otimes B^\ell\rightarrow
	\mathcal H^0(S, B^k)\otimes M^{\otimes (r-1)}_{B^k}\otimes B^\ell
	\rightarrow M^{\otimes (r-1)}_{B^k}\otimes B^{k+\ell}\rightarrow 0
	\,\,\,\,\,\,\,\,\,\,\,\,\,\,\,\,\,(\star)
$$
for $k, r\geq 1$ and $\ell\geq 0$.

\begin{theo}\label{vanishingthm}
Let $B$ be an ample line bundle on a fake projective plane $S$ such that $B\otimes K_S^{-2}$ is nef. Let $k\geq 1$ and $\ell\geq 0$. Then, 
\\
\\
(i) $H^1(S, M_{B^k}\otimes B^\ell)=0$ \,\,\,\,if\,\,\,\, $k+\ell\geq 3$;\\
(ii) $H^1(S, M^{\otimes r}_{B^k}\otimes B^\ell)=0$ \,\,\,\,if\,\,\,\, $k+\ell\geq r+2$ and $\ell\geq r\geq 2$;\\
(iii) $H^2(S, M^{\otimes r}_{B^k}\otimes B^\ell)=0$ \,\,\,\,if\,\,\,\, $\ell+1\geq r\geq 1$.
\end{theo}

\begin{proof}
We first show that 
$H^1(S, M_{B^k}\otimes B^\ell)=0$ if $k+\ell\geq 3$.  In order to simplify the notations, in the proof we will skip the reference to $S$ when writing the cohomology groups.

From the long exact sequence associated to $(\star)$, we get the exact sequence
$$
\begin{matrix}
	\cdots\rightarrow H^0(\mathcal H^0(S,B^k)\otimes B^\ell)\rightarrow H^0(B^{k+\ell})\\ 		   \rightarrow H^1(M_{B^k}\otimes B^\ell)
	\rightarrow H^1(\mathcal H^0(S,B^k)\otimes B^\ell)
	\rightarrow\cdots
\end{matrix}
$$

As $H^1(\mathcal H^0(S,B^k)\otimes B^\ell)\cong H^0(B^k)\otimes H^1(B^\ell)$, it vanishes  by Proposition~\ref{vanishingprop}. Thus, $H^1(M_{B^k}\otimes B^\ell)$ vanishes if the map $H^0(\mathcal H^0(S,B^k)\otimes B^\ell)\rightarrow H^0(B^{k+\ell})$ is surjective. Equivalently, we need the surjectivity for the map
$$
	H^0(B^k)\otimes H^0(B^\ell)\rightarrow H^0(B^{k+\ell})
	\,\,\,\,\,\,\,\,\,\,\,\,\,\,\,\,\,\,\,\,\,\,\,\,\,\,\,\,\,\,\,\,(\sharp)
$$
since $H^0(\mathcal H^0(S,B^k)\otimes B^\ell)\cong H^0(B^k)\otimes H^0(B^\ell)$. Consider the natural maps 
$$
H^0(B)\otimes\cdots\otimes H^0(B)\otimes H^0(B^\ell)\overset{p}{\longrightarrow} H^0(B^k)\otimes H^0(B^\ell)\overset{q}{\longrightarrow} H^0(B^{k+\ell}).
$$
We see that it suffices to have the surjectivity for $q\circ p$, which is in turn implied by the surjectivity of all the mappings 
$$H^0(B)\otimes H^0(B^{m-1+\ell})\rightarrow H^0(B^{m+\ell})$$ 
for $1\leq m\leq k$.
We now apply the result of Castelnuovo-Mumford (Theorem~\ref{regularity}). Thus, the above maps are surjective if
$$
	H^1(B^{\ell-1})=H^2(B^{\ell-2})=0.
$$
By Proposition~\ref{vanishingprop}, these are satisfied if $\ell\geq 2$.
(Here we remark that all the higher cohomology groups appeared in Castelnuovo-Mumford Theorem vanish because $\dim(S)=2$.)  
Thus, we see that for $k\geq 1$, $H^1(S, M_{B^k}\otimes B^\ell)=0$ if $\ell\geq 2$. Since the roles of $k$ and $\ell$ are symmetric for the mapping in $(\sharp)$, it follows that for $k\geq 2$, we also have $H^1(S, M_{B^k}\otimes B^\ell)=0$ for $\ell=1$. We hence deduce that $H^1(S, M_{B^k}\otimes B^\ell)=0$ if $k+\ell\geq 3$. Note that we can also allow $\ell=0$ here as the map in $(\sharp)$ is surjective for $\ell=0$. We have thus proven $(i)$.

To prove $(ii)$ and $(iii)$, we will need to first establish $H^2(M_{B^k}\otimes B^\ell)=0$ for $\ell\geq 0$ and $H^2(M^{\otimes 2}_{B^k}\otimes B^\ell)=0$ for $\ell\geq 1$. We proceed further along the long exact sequence associated to $(\star)$, 
$$
\begin{matrix}
	\cdots\rightarrow H^1( M^{\otimes(r-1)}_{B^k}\otimes B^{k+\ell}) \\ 
	\rightarrow H^2(M^{\otimes r}_{B^k}\otimes B^\ell)
	\rightarrow H^2(\mathcal H^0(S,B^k) \otimes M^{\otimes(r-1)}_{B^k}\otimes B^\ell)
	\rightarrow\cdots
\end{matrix}
$$ 

Here in the sequence, for $r=1$, we see that the first term $H^1(B^{k+\ell})$ is zero by Proposition~\ref{vanishingprop} and for the last term, we have $H^2(\mathcal H^0(S,B^k)\otimes B^\ell)\cong H^0(B^k)\otimes H^2(B^\ell)=0$. It now follows that $H^2(M_{B^k}\otimes B^\ell)=0$ for $\ell\geq 0$. For $r=2$, the first term $H^1( M_{B^k}\otimes B^{k+\ell})$ vanishes if $2k+\ell\geq 3$ by $(i)$. For the last term, 
$$
H^2(\mathcal H^0(S,B^k) \otimes M_{B^k}\otimes B^\ell)
\cong H^0(B^k)\otimes H^2(M_{B^k}\otimes B^\ell)=0
$$
if $\ell\geq 0$. Therefore, we have $H^2(M^{\otimes 2}_{B^k}\otimes B^\ell)=0$ if $\ell\geq 1$.

Now we prove the $(ii)$ and $(iii)$ for $r\geq 2$ by induction on $r$.

Consider now the case $r=2$. From the long exact sequence associated to $(\star)$ above, we have the exact sequence
$$
\begin{matrix}
	\cdots\rightarrow H^0(\mathcal H^0(S,B^k)\otimes M_{B^k}\otimes B^\ell)\rightarrow H^0( M_{B^k}\otimes B^{k+\ell}) \\ 
	\rightarrow H^1(M^{\otimes 2}_{B^k}\otimes B^\ell)
	\rightarrow H^1(\mathcal H^0(S,B^k) \otimes M_{B^k}\otimes B^\ell)
	\rightarrow\cdots
\end{matrix}
$$

The last term 
$$H^1(\mathcal H^0(S,B^k) \otimes M_{B^k}\otimes B^\ell)\cong
H^0(B^k)\otimes H^1(M_{B^k}\otimes B^\ell)$$ 
vanishes when $k+\ell\geq 3$. Hence, for $k+\ell\geq 3$, the cohomology $H^1(M^{\otimes 2}_{B^k}\otimes B^\ell)$ vanishes if 
$$
H^0(\mathcal H^0(S,B^k)\otimes M_{B^k}\otimes B^\ell)\rightarrow H^0( M_{B^k}\otimes B^{k+\ell}) 
$$
is surjective, which is equivalent to the surjectivity of
$$
H^0(B^k)\otimes H^0(M_{B^k}\otimes B^\ell)\rightarrow H^0( M_{B^k}\otimes B^{k+\ell}) 
$$
since
$$
H^0(\mathcal H^0(S,B^k)\otimes M_{B^k}\otimes B^\ell)
\cong H^0(B^k)\otimes H^0(M_{B^k}\otimes B^\ell).
$$

From the Castelnuovo-Mumford Theorem again, as in the case for $r=1$, the above map is surjective if we have
$$
H^1(M_{B^k}\otimes B^{\ell-1})
=H^2(M_{B^k}\otimes B^{\ell-2})=0.
$$
By what we have proven previously, these are satisfied if $\ell\geq 2$ and $k+\ell\geq 4$.

Suppose now $r\geq 3$. From $(\star)$, we have the exact sequence
$$
\begin{matrix}
	\cdots\rightarrow H^0(\mathcal H^0(S,B^k)\otimes M^{\otimes(r-1)}_{B^k}\otimes B^\ell)\rightarrow H^0( M^{\otimes(r-1)}_{B^k}\otimes B^{k+\ell}) \\ 
	\rightarrow H^1(M^{\otimes r}_{B^k}\otimes B^\ell)
	\rightarrow H^1(\mathcal H^0(S,B^k) \otimes M^{\otimes(r-1)}_{B^k}\otimes B^\ell)
	\rightarrow\cdots
\end{matrix}
$$

By the induction hypothesis, the last term 
$$H^1(\mathcal H^0(S,B^k) \otimes M^{\otimes(r-1)}_{B^k}\otimes B^\ell)\cong
H^0(B^k)\otimes H^1(M^{\otimes(r-1)}_{B^k}\otimes B^\ell)$$ 
vanishes when $k+\ell\geq r$ and $\ell\geq r$. Hence, under the same conditions on $k,\ell,r$, the cohomology $H^1(M^{\otimes r}_{B^k}\otimes B^\ell)$ vanishes if 
$$
H^0(\mathcal H^0(S,B^k)\otimes M^{\otimes(r-1)}_{B^k}\otimes B^\ell)\rightarrow H^0( M^{\otimes(r-1)}_{B^k}\otimes B^{k+\ell}) 
$$
is surjective, which is equivalent to the surjectivity of
$$
H^0(B^k)\otimes H^0(M^{\otimes(r-1)}_{B^k}\otimes B^\ell)\rightarrow H^0( M^{\otimes(r-1)}_{B^k}\otimes B^{k+\ell}) 
$$
since
$$
H^0(\mathcal H^0(S,B^k)\otimes M^{\otimes(r-1)}_{B^k}\otimes B^\ell)
\cong H^0(B^k)\otimes H^0(M^{\otimes(r-1)}_{B^k}\otimes B^\ell).
$$

We argue similarly with the Castelnuovo-Mumford Theorem. Thus, the above map is surjective if we have
$$
H^1(M^{\otimes(r-1)}_{B^k}\otimes B^{\ell-1})
=H^2(M^{\otimes(r-1)}_{B^k}\otimes B^{\ell-2})=0.
$$

The induction hypothesis implies that these conditions are satisfied if $k+(\ell-1)\geq (r-1)+2$ and $\ell-1\geq r-1$ and $\ell-2\geq (r-1)-1$. In particular, it suffices to have $k+\ell\geq r+2$ and $\ell\geq r$.

It remains to consider $H^2(M^{\otimes r}_{B^k}\otimes B^\ell)$. 
From the long exact sequence associated to $(\star)$, we get the exact sequence
$$
\begin{matrix}
	\cdots\rightarrow H^1( M^{\otimes(r-1)}_{B^k}\otimes B^{k+\ell}) \\ 
	\rightarrow H^2(M^{\otimes r}_{B^k}\otimes B^\ell)
	\rightarrow H^2(\mathcal H^0(S,B^k) \otimes M^{\otimes(r-1)}_{B^k}\otimes B^\ell)
	\rightarrow\cdots
\end{matrix}
$$ 

By the induction hypothesis, we have $H^1( M^{\otimes(r-1)}_{B^k}\otimes B^{k+\ell})=0$ for $2k+\ell\geq r+1$ and $k+\ell\geq r-1$. Also by the induction hypothesis, $$H^2(\mathcal H^0(S,B^k) \otimes M^{\otimes(r-1)}_{B^k}\otimes B^\ell)\cong H^0(B^k)\otimes H^2(M^{\otimes(r-1)}_{B^k}\otimes B^\ell)=0$$ 
if $\ell+1\geq r-1$. Combining all these, it now follows that $H^2(M^{\otimes r}_{B^k}\otimes B^\ell)=0$ if $\ell+1\geq r$. The proof is now complete.
\end{proof}

 \noindent{\bf Remark.}
As in the proof of Proposition~\ref{vanishingprop}, by applying the result of Reider on a fake projective plane $S$, it follows easily that for any ample line bundle $A$, its power $A^k$ is base-point-free for $k\geq 6$ and hence the previous theorem will also imply certain vanishing statements for $A$. However, if we just simply do the direct translation in this way, there will be ``gaps'' between the powers of $A$ in which we do not know whether the analogous cohomology groups vanish or not.

\medskip
 In view of the remark above, for the purpose of obtaining a lower bound, above which every power of an arbitrary ample line bundle on $S$ satisfies the property $N_p$, we improve upon the proof of Theorem~\ref{vanishingthm} and obtain the following

\begin{theo}\label{vanishingthm2}
Let $A$ be an ample line bundle and $\sigma$, $\tau$ be torsion line bundles on a fake projective plane $S$. For $r\geq 1$ and $k\geq 12$, we have $H^1(S, M^{\otimes r}_{A^k\otimes\sigma}\otimes A^\ell\otimes\tau)=0$ if $\ell\geq 6(r+2)$, and $H^2(S, M^{\otimes r}_{A^k\otimes\sigma}\otimes A^\ell\otimes\tau)=0$ if $\ell\geq 6r-2$.
\end{theo}

\begin{proof}
The structure of the proof will be similar to that of Theorem~\ref{vanishingthm} and thus we will be brief on certain arguments.
We will again prove by induction on $r$. 

Suppose $r=1$. When $\ell\geq 18$, as $H^1(A^\ell\otimes\tau)=0$, the cohomology $H^1(M_{A^k\otimes\sigma}\otimes A^\ell\otimes\tau)$ vanishes if the map
$$
	H^0(A^k\otimes\sigma)\otimes H^0(A^\ell\otimes\tau)\rightarrow H^0(A^{k+\ell}\otimes\sigma\otimes\tau)
$$
is surjective. 

Since $k\geq 12$,  we can write $A^k\otimes\sigma=A^{k_1}\otimes \cdots\otimes A^{k_p}\otimes\sigma$ for some $p\geq 2$ such that $6\leq k_1\leq k_2\leq\cdots\leq k_p$. When $k=12$ or $13$, such decomposition is unique. If $14\leq k\leq 17$, we choose $k_1=7$ and hence $7\leq k_2\leq 10$. If $k\geq 18$, we can always choose our $k_j$ such that $6=k_1\leq\cdots\leq k_p\leq 9$ for some $p\geq 3$. Hence, for any $k\geq 12$, we can always choose $k_j$ such that $6\leq k_1\leq 7$ and $k_1\leq\cdots\leq k_p\leq 10$.

By considering the natural maps 
$$
H^0(A^{k_1})\otimes\cdots\otimes H^0(A^{k_p}\otimes\sigma)\otimes H^0(A^\ell\otimes\tau)\longrightarrow H^0(A^k\otimes\sigma)\otimes H^0(A^\ell\otimes\tau)\longrightarrow H^0(A^{k+\ell}\otimes\sigma\otimes\tau),
$$
we see that it suffices to have the surjectivity for the naturally defined mappings 
$$
\varphi_j: H^0(A^{k_j})\otimes H^0(A^{k_1+\cdots+k_{j-1}+\ell}\otimes\tau)\rightarrow H^0(A^{k_1+\cdots+k_{j}+\ell}\otimes\tau)
$$ 
for $1\leq j\leq p-1$ and also
$$
\varphi_p: H^0(A^{k_p}\otimes\sigma)\otimes H^0(A^{k_1+\cdots+k_{p-1}+\ell}\otimes\tau)\rightarrow H^0(A^{k+\ell}\otimes\sigma\otimes\tau).
$$

Since $k_j\geq 6$, every $A_{k_j}$ and also $A_{k_p}\otimes\sigma$ are base-point-free as in the proof of Proposition~\ref{vanishingprop}. 
Thus, we can apply the result of Castelnuovo-Mumford and it
says that $\varphi_1$ is surjective if
$$
	H^1(A^{\ell-k_1}\otimes\tau)=H^2(A^{\ell-2k_1}\otimes\tau)=0.
$$
By Lemma~\ref{vanishinglem}, these are true if $\ell\geq 18$ since $6\leq k_1\leq 7$. Similarly, for $2\leq j\leq p-1$, $\varphi_j$ is surjective if
$$
	H^1(A^{k_1+\cdots+k_{j-1}+\ell-k_j}\otimes\tau)=H^2(A^{k_1+\cdots+k_{j-1}+\ell-2k_j}\otimes\tau)=0.
$$
These hold true if $\ell\geq 18$ since $k_1\geq 6$ and $k_j\leq 10$. The surjectivity for $\varphi_p$ can be seen in the same manner.

For $H^2(M_{A^k\otimes\sigma}\otimes A^\ell\otimes\tau)$, it vanishes for $\ell\geq 4$, which can be seen easily from the long exact sequence associated to $(\star)$ for $r=1$, 
$$
\begin{matrix}
	\cdots\rightarrow H^1(A^{k+\ell}\otimes\sigma\otimes\tau)
	\rightarrow H^2(M_{A^k\otimes\sigma}\otimes A^\ell\otimes\tau)\rightarrow H^2(\mathcal H^0(S,A^k\otimes\sigma)\otimes A^\ell\otimes\tau)
	\rightarrow\cdots
\end{matrix}
$$
The case for $r=1$ is now settled.

Suppose now $r\geq 2$. We first consider $H^1(M^{\otimes r}_{A^k}\otimes A^\ell)$. We have the exact sequence
$$
\begin{matrix}
	\cdots\rightarrow H^0(\mathcal H^0(S,A^k\otimes\sigma)\otimes M^{\otimes(r-1)}_{A^k\otimes\sigma}\otimes A^\ell\otimes\tau)\rightarrow H^0( M^{\otimes(r-1)}_{A^k\otimes\sigma}\otimes A^{k+\ell}\otimes\sigma\otimes\tau) \\ 
	\rightarrow H^1(M^{\otimes r}_{A^k\otimes\sigma}\otimes A^\ell\otimes\tau)
	\rightarrow H^1(\mathcal H^0(S,A^k\otimes\sigma) \otimes M^{\otimes(r-1)}_{A^k\otimes\sigma}\otimes A^\ell\otimes\tau)
	\rightarrow\cdots
\end{matrix}
$$
By the induction hypothesis, the last term 
$$
H^1(\mathcal H^0(S,A^k\otimes\sigma) \otimes M^{\otimes(r-1)}_{A^k\otimes\sigma}\otimes A^\ell\otimes\tau)
\cong
H^0(A^k\otimes\sigma)\otimes H^1(M^{\otimes(r-1)}_{A^k\otimes\sigma}\otimes A^\ell\otimes\tau)$$ 
vanishes when $\ell\geq 6(r+1)$. Hence, for $\ell\geq 6(r+1)$, the cohomology group $H^1(M^{\otimes r}_{A^k\otimes\sigma}\otimes A^\ell\otimes\tau)$ vanishes if 
$$
H^0(A^k\otimes\sigma)\otimes H^0(M^{\otimes(r-1)}_{A^k\otimes\sigma}\otimes A^\ell\otimes\tau)\rightarrow H^0( M^{\otimes(r-1)}_{A^k\otimes\sigma}\otimes A^{k+\ell}\otimes\sigma\otimes\tau) 
$$
is surjective. As in the case $r=1$, we can write $A^k\otimes\sigma=A^{k_1}\otimes \cdots\otimes A^{k_p}\otimes\sigma$ for some $p\geq 2$ such that $k_j\geq 6$ for $1\leq j\leq p$. But for $r\geq 2$, we choose our $k_j$ such that $6=k_1\leq k_2\leq\cdots\leq k_p\leq 11$. 

As before,
it suffices to have the surjectivity for
$$
\phi_j: H^0(A^{k_j})\otimes H^0(M^{\otimes(r-1)}_{A^k\otimes\sigma}\otimes A^{k_1+\cdots+k_{j-1}+\ell}\otimes\tau)\rightarrow H^0(M^{\otimes(r-1)}_{A^k\otimes\sigma}\otimes A^{k_1+\cdots+k_{j}+\ell}\otimes\tau)
$$ 
for $1\leq j\leq p-1$ and also
$$
\phi_p: H^0(A^{k_p}\otimes\sigma)\otimes H^0(M^{\otimes(r-1)}_{A^k\otimes\sigma}\otimes A^{k_1+\cdots+k_{p-1}+\ell}\otimes\tau)\rightarrow H^0(M^{\otimes(r-1)}_{A^k\otimes\sigma}\otimes A^{k+\ell}\otimes\sigma\otimes\tau).
$$
By Castelnuovo-Mumford Theorem, $\phi_1$ is surjective if
$$
H^1(M^{\otimes(r-1)}_{A^k\otimes\sigma}\otimes A^{\ell-k_1}\otimes\tau)
=H^2(M^{\otimes(r-1)}_{A^k\otimes\sigma}\otimes A^{\ell-2k_1}\otimes\tau)=0.
$$

The induction hypothesis implies that these conditions are satisfied if $\ell-k_1=\ell-6\geq 6(r+1)$ and $\ell-2k_1=\ell-12\geq 6(r-1)-2$, which are both true if $\ell\geq 6(r+2)$.

For $2\leq j\leq p-1$, $\phi_j$ is surjective if
$$
	H^1(M^{\otimes(r-1)}_{A^k\otimes\sigma}\otimes A^{k_1+\cdots+k_{j-1}+\ell-k_j}\otimes\tau)=H^2(M^{\otimes(r-1)}_{A^k\otimes\sigma}\otimes A^{k_1+\cdots+k_{j-1}+\ell-2k_j}\otimes\tau)=0.
$$
By simple arithmetic, together with the induction hypothesis, it follows that these are satisfied if $\ell-5\geq 6(r+1)$ and $\ell-16\geq 6(r-1)-2$. In particular, it is more than sufficient if we have $\ell\geq 6(r+2)$. The surjectivity for $\phi_p$ is checked similarly.

It remains to check that $H^2(M^{\otimes r}_{A^k\otimes\sigma}\otimes A^\ell\otimes\tau)=0$ when $\ell\geq 6r-2$. 
From the exact sequence
$$
\begin{matrix}
	\cdots\rightarrow H^1( M^{\otimes(r-1)}_{A^k\otimes\sigma}\otimes A^{k+\ell}\otimes\sigma\otimes\tau) \\ 
	\rightarrow H^2(M^{\otimes r}_{A^k\otimes\sigma}\otimes A^\ell\otimes\tau)
	\rightarrow H^2(\mathcal H^0(S,A^k\otimes\sigma) \otimes M^{\otimes(r-1)}_{A^k\otimes\sigma}\otimes A^\ell\otimes\tau)
	\rightarrow\cdots,
\end{matrix}
$$ 
we have, by the induction hypothesis, $H^1( M^{\otimes(r-1)}_{A^k\otimes\sigma}\otimes A^{k+\ell}\otimes\sigma\otimes\tau)=0$ for if $k+\ell\geq 6(r+1)$. Since $k\geq 12$, the inequality is satisfied when $\ell\geq 6r-2$. 
Similarly, 
$$
H^2(\mathcal H^0(S,A^k\otimes\sigma) \otimes M^{\otimes(r-1)}_{A^k\otimes\sigma}\otimes A^\ell\otimes\tau)\cong H^0(A^k\otimes\sigma)\otimes H^2(M^{\otimes(r-1)}_{A^k\otimes\sigma}\otimes A^\ell\otimes\tau)=0
$$ 
if $\ell\geq 6(r-1)-2$. Thus, the second term $H^2(M^{\otimes r}_{A^k}\otimes A^\ell)$ in the exact sequence above is zero when $\ell\geq 6r-2$. The proof is now complete.
\end{proof}

We are now ready to obtain the final result of this section and also give the proof of Theorem~\ref{mainthm2}.

\begin{theo}\label{mainthm}
Let $B$ be an ample line bundle on a fake projective plane $S$ such that $B\otimes K_S^{-2}$ is nef, where $K_S$ is the canonical line bundle. Then, $B^m$ is projectively normal (i.e. satisfies the property $N_0$) if $m\geq 2$. More generally, for $p\geq 1$, $B^m$ satisfies the property $N_p$ if $m\geq p+1$. 
\end{theo}

\begin{proof} As remark earlier, if $B\otimes K_S^{-2}$ is nef, then $B^m$ is base-point-free for $m\geq 1$. Suppose $m\geq 2$. For $B^m$ to satisfy the property $N_0$, by Green's criterion, it suffices to have $H^1(S, B^{ms})=H^1(S, M_{B^m}\otimes B^{ms})=0$ for  $s\geq 1$. These follow respectively from Proposition~\ref{vanishingprop} and Theorem~\ref{vanishingthm}.

Suppose now $p\geq 1$ and $m\geq p+1$. To see that $B^m$ satisfies the property $N_p$, we just need to check $H^1(S, M^{\otimes r}_{B^m}\otimes B^{ms})=0$ for $0\leq r\leq p+1$ and $s\geq 1$. When  $r=0$, it again follows from Proposition~\ref{vanishingprop}. For $r=1$, it follows from Theorem~\ref{vanishingthm}~$(i)$ since $m+ms\geq 2m\geq 2(p+1)\geq 4$. Finally, for $2\leq r\leq p+1$, we have 
$$
m+ms\geq 2m\geq 2(p+1)\geq 2r\geq r+2,
$$
and $ms\geq p+1\geq r\geq 2$ and hence the result follows from Theorem~\ref{vanishingthm}~$(ii)$.
\end{proof}

\begin{coro}
In addition to the results for pluricanonical bundles given by Corollary~\ref{maincoro}, we also have
$K^{2p+2}_S$  satisfies the property $N_p$ for $p\geq 1$. In particular, $K^{2p+4}_S$ satisfies the property $N_p$ for  $p\geq 0$.
\end{coro}

\noindent\textbf{Remark.} For an ample and base-point-free line bundle $B$ on a complex projective surface of general type, Gallego-Purnaprajna~[GP] proved that $B^m$ satisfies the property $N_1$ if $m\geq 3$; and for $p\geq 2$, $B^m$ satisfies the property $N_p$ if $m\geq p+1$. The latter result contains in particular our results for $p\geq 2$ in Theorem~\ref{mainthm} since the line bundles satisfying the hypotheses of Theorem~\ref{mainthm} are ample and base-point-free.

\begin{proof}[Proof of Theorem~\ref{mainthm2}] Let $m$ be a positive integer such that $m\geq 6(p+3)$. We fix an ample generator $G$ of Pic$(S)$ and let $A$ and $N$ be an ample and a nef line bundle on $S$ respectively. To show that $A^m\otimes N$ satisfies the property $N_p$, we  will verify for $0\leq r\leq p+1$ and $s\geq 1$, $H^1(S, M^{\otimes r}_{A^m\otimes N}\otimes (A^m\otimes N)^s)=0$ . 

We have $A^m\otimes N=G^k\otimes\sigma$, for some positive integer $k\geq 6(p+3)$ and some torsion line bundle $\sigma$. Thus, $H^1(S, (A^m\otimes N)^s)=H^1(S, G^{ks}\otimes\sigma^s)=0$ for $s\geq 1$ by Lemma~\ref{vanishinglem}. Next, for $1\leq r\leq p+1$ and $s\geq 1$, 
$$
	H^1(S, M^{\otimes r}_{A^m\otimes N}\otimes (A^m\otimes N)^s)
	=H^1(S, M^{\otimes r}_{G^k\otimes\sigma}\otimes G^{ks}\otimes\sigma^s)=0
$$
by Theorem~\ref{vanishingthm2} since $k\geq 18$ and $ks\geq k\geq 6(p+3)\geq 6(r+2)$.
\end{proof}

\section{Syzygies for base-point-free line bundles}

Let $S$ be a fake projective plane. There is a one-to-one correspondence between set of torsion elements of $H_1(S,\bZ)$ and the set of torsion line bundles on $S$, as given by the Universal Coefficient Theorem (cf. [LY]). Moreover, a torsion line bundle $\delta$ corresponds to a surjective representation $\rho_\delta:\pi_1(S)\rightarrow \mathbb Z/d\mathbb Z$, where $d\in\mathbb N^+$ is the order of $\delta$. The representation $\rho_\delta$ necessarily descends to a representation $\rho^\flat_\delta:H_1(S,\mathbb Z)\cong\pi_1(S)/[\pi_1(S),\pi_1(S)]\rightarrow \mathbb Z/d\mathbb Z$. If we let $C_\delta:=\ker(\rho^\flat_\delta)$, then $H_1(S,\mathbb Z)/C_\delta\cong\mathbb Z/d\mathbb Z$.

\begin{lemm}\label{h1lemma}
Let $M$ be a compact complex manifold with $b_1(M)=0$ and $\delta$ be a torsion line bundle on $M$ of order $d$. Let $M_\delta$ be the finite unramified covering of $M$ defined by the kernel of the representation $\rho_\delta:\pi_1(M)\rightarrow \mathbb Z/d\mathbb Z$ associated to $\delta$. Then $h^1(M,\delta)=0$ if $b_1(M_\delta)=0$.
\end{lemm}
\begin{proof}  
Let $\pi:M_\delta\rightarrow M$ be the finite unramified covering map. Then $\pi^*\delta$ is the trivial line bundle on $M_\delta$ and we have (see~[BHPV, p.55]) 
$$\pi_*(\cO_{M_\delta})=\bigoplus^{d-1}_{j=0}\cO_M(\delta^j).$$
Hence,
$$\frac12 b_1(M_\delta)=h^1(M_\delta,\cO_{M_\delta})=h^1(M,\pi_*(\cO_{M_\delta}))=\sum^{d-1}_{j=0}h^1(M,\delta^j).$$
The lemma now follows.
\end{proof}

Let $S$ be a fake projective plane and $\Pi\subset PU(2,1)$ be the lattice associated to $S$, which is isomorphic to the fundamental group $\pi_1(S)$. By definition, $b_1(S)=0$ as $S$ is a fake projective plane and it follows that $H_1(S,\bZ)\cong \Pi/[\Pi,\Pi]$ is a finite abelian group. Let $\rho:\Pi\rightarrow\Pi/[\Pi,\Pi]$ be the canonical projection. Let $C$ be a subgroup of $\Pi/[\Pi,\Pi]$ and $p_C:\Pi/[\Pi,\Pi]\rightarrow (\Pi/[\Pi,\Pi])/C$ be the natural projection and $\rho_C=p_C\circ\rho$.  Let $S_C=B_{\bC}^2/\ker(\rho_C)$ be the finite unramified covering of $S$ associated to $\ker(\rho_C)$, where $B_{\bC}^2\subset\mathbb C^2$ is the unit ball.

%Hence $H_1(M,\bZ)=\oplus_a C_a$ is a finite direct sum of cyclic groups $C_a$ of finite order.  Let $\rho_o:\Pi\rightarrow \Pi/[\Pi,\Pi]$
%be the abelianization map.  Let $p_{C_a}:\Pi/[\Pi,\Pi]\rightarrow C_a$ be the projection to the $a$-th factor.  In general, if $C$ is a finite subgroup of $\oplus_a C_a$,

%% We need the following result of Donald Cartwright, who kindly provided the proof at our request.
%% \begin{theo}[Cartwright]\label{cartwright}
%% For every subgroup $C$ of $\Pi/[\Pi,\Pi]$, we have $b_1(M_C)=0$.
%% \end{theo}
%% We can make the following reduction for the theorem above. Let $M_0$ be the unramified covering of $M$ corresponding to the lattice $[\Pi,\Pi]$. As $\ker(\rho)=[\Pi,\Pi]$ and $\ker(\rho)\subset\ker(\rho_C)$, it is clear that
%% $M_0$ is a finite unramified covering of $M_C$, which is also a holomorphic isometric covering with respect to the canonical metrics induced from $B_{\bC}^2$. In particular, any harmonic one form on $B_{\bC}^2/\ker(\rho_C)$ can be pulled back to
%% $B_{\bC}^2/\ker(\rho_o)$ and hence $b_1(M_1)\geqslant b_1(M_C)$ from Poincar\'e Duality.
%% Hence to prove the Theorem~\ref{cartwright}, it suffices to check that
%% $b_1(M_1)=0$.  This was verified by Cartwright using Magma, which will be explained in the Appendix. 

\begin{theo}\label{cartwright}
For every subgroup $C$ of $\Pi/[\Pi,\Pi]$, we have $b_1(S_C)=0$.
\end{theo}
\begin{proof}
Let $S_0$ be the unramified covering of $S$ corresponding to the lattice $[\Pi,\Pi]$. As $\ker(\rho)=[\Pi,\Pi]$ 
and $\ker(\rho)\subset\ker(\rho_C)$, it is clear that $S_0$ is a finite unramified covering of $S_C$, which is 
also a holomorphic isometric covering with respect to the canonical metrics induced from $B_{\bC}^2$. In 
particular, any harmonic one form on $B_{\bC}^2/\ker(\rho_C)$ can be pulled back to $B_{\bC}^2/\ker(\rho)$ 
and hence $b_1(S_0)\geqslant b_1(S_C)$ from Poincar\'e Duality.
Hence it suffices to check that $b_1(S_0)=0$. Equivalently, we must check that
the abelianization $\Pi'/[\Pi',\Pi']$ of $\Pi'=[\Pi,\Pi]$ is finite. This was 
kindly verified for us by Donald Cartwright using Magma, as follows. There are (see~[CS]) exactly~50
groups~$\Pi$'s which are fundamental groups of fake projective planes~$S$. Their names are
listed in the file \verb'registerofgps.txt' in the weblink of [CS]). Each is a subgroup
of a maximal arithmetic subgroup $\bar\Gamma$ of~$PU(2,1)$. Explicit generators and relations
for these $\bar\Gamma$'s are given in various files of the weblink. For example, see \verb'C20p2/gpc20p2generators_reducesyntax.txt'
for the generators and relations for the three $\bar\Gamma$'s named $(C_{20},p=2,\emptyset)$,
$(C_{20},p=2,\{3+\})$ and $(C_{20},p=2,\{3-\})$. Abstract presentations of these groups $\bar\Gamma$ are given
in the file {\tt barGammapresentations.txt} of that weblink, where for each of the fake projective planes~$S$ 
named in {\tt registerofgps.txt}, generators $s_1,s_2,\ldots$ are given for its fundamental group~$\Pi$. 

Using the data about generators and relations for the lattices involved from the file {\tt finite-cover-b1.txt}, for each of these 50 $\Pi$'s, we exhibit the subgroup $\Pi'=[\Pi,\Pi]$ of~$\Pi$, 
then verify that $\Pi'/[\Pi',\Pi']$ is finite using Magma's {\tt AbelianQuotientInvariants\/} routine. 

The subgroup~$\Pi'$ was in each case found as follows. Using Magma's {\tt Rewrite\/}
routine, we first find a presentation of~$\Pi$ from that of $\bar\Gamma$. Let
$n=|\Pi/[\Pi,\Pi]|=|H_1(S,\bZ)|$. Using Magma's {\tt LowIndexNormalSubgroups\/}$(\Pi,n)$ 
routine, we list the normal subgroups of~$\Pi$ of index at most~$n$. Then we find the
unique subgroup in this (sometimes lengthy) list which has index exactly~$n$, and 
contains all commutators $s_i^{-1}s_j^{-1}s_is_j$.
\end{proof}

As discussed at the beginning of this section, for each torsion line bundle $\delta$ of order $d$ on a fake projective plane $S=B^2_{\bC}/\Pi$, there is a subgroup $C_\delta<H_1(S,\mathbb Z)$ such that $H_1(S,\mathbb Z)/C_\delta\cong\mathbb Z/d\mathbb Z$ and the representation of $\pi_1(S)$ associated to $\delta$ is just the composition $\pi_1(S)\rightarrow H_1(S,\mathbb Z)\rightarrow H_1(S,\mathbb Z)/C_\delta$. With the identifications $\pi_1(S)\cong\Pi$ and $H_1(S,\mathbb Z)\cong\pi_1(S)/[\pi_1(S),\pi_1(S)]\cong\Pi/[\Pi,\Pi]$, the finite unramified cover $M_\delta$ associated to $\delta$ described in Lemma~\ref{h1lemma} is just $S_C$ in Theorem~\ref{cartwright} when $C=C_\delta$. Thus, by combining Lemma~\ref{h1lemma} and Theorem~\ref{cartwright}, we have

\begin{coro}\label{h1corollary}
Let $S$ be a fake projective plane and $\delta$ be a torsion line bundle on $S$. Then $h^1(S,\delta)=0$.
\end{coro}

The following lemma is a consequence of a well known theorem of Remmert-Van de Ven~[RV, p.155]. One may also see~[K, 15.6.2].

\begin{lemm}\label{dimensionlemma}
Let $M$ be a complex projective manifold and $L$ be a holomorphic line bundle on $M$ such that $h^0(M,L)\geq 1$. Then $h^0(M, L^2)\geq 2h^0(M, L)-1$.
\end{lemm}

\begin{prop}\label{dimensionprop}
Let $S$ be a fake projective plane and $G$ be an ample generator of $\textrm{Pic}(S)$. For any torsion line bundle $\delta$ on $S$, $h^0(S, G\otimes\delta)\leq 1$, $h^0(S, G^2\otimes\delta)\leq 2$ and $h^0(S, G^3\otimes\delta)\leq 1$.
\end{prop}
\begin{proof}
Since $K_S=G^3\otimes\tau$ for some torsion line bundle $\tau$, using Riemann-Roch Formula with Kodaira Vanishing Theorem, it is easy to get that $h^0(G^4\otimes\delta^2)=3$ for any torsion line bundle $\delta$. It now follows from Lemma~\ref{dimensionlemma} that $h^0(S, G^2\otimes\delta)\leq 2$, which in turn implies $h^0(S, G\otimes\delta)\leq 1$.

For $h^0(S, G^3\otimes\delta)$, the result follows again from Riemann-Roch Formula, together with Corollary~\ref{h1corollary} and the fact that $K_S=G^3\otimes\tau$.
\end{proof}

\begin{proof}[Proof of Proposition 1]  This is now a consequence of Corollary  3 and Proposition 3.
\end{proof}

We are now ready to prove Theorem~\ref{mainthm3}.

\begin{proof}[Proof of Theorem~\ref{mainthm3}]
Let $G$ be an ample generator of $\textrm{Pic}(S)$. As the Picard number of $S$ is 1, $B$ must be ample, and we can write $B=G^k\otimes\sigma$, for some $k\geq 1$ and some torsion line bundle $\sigma$. As $B$ is base-point-free, it is clear from Proposition~\ref{dimensionprop} that $k\neq 1, 3$. We claim that $k\neq 2$ neither.  Suppose otherwise that $k=2$.  Then by Proposition~\ref{dimensionprop}, $h^0(S,B)\leq 2$.  Since $B$ is base-point-free, we have in fact $h^0(S,B)=2$.  But $S$ is of Picard number 1, the zero divisors of two linearly independent sections of $\Gamma(S,B)$ must have non-trivial intersection, which contradicts base-point-freeness.  Hence, $k\geq 4$.
Finally, as Lemma~\ref{vanishinglem} is true whenever $k\geq 4$, we conclude that Proposition~\ref{vanishingprop}, Theorem~\ref{vanishingthm} and Theorem~\ref{mainthm} are all true for any base-point-free line bundle.
\end{proof}

\section{Refinement with a conjecture on vanishing cohomology}

Motivated by the existence problem for exceptional collections for objects in the derived category of coherent sheaves on a fake projective plane $S$, we formulate the following conjecture, which is a variant of  Conjecture 1.1  in~[GKMS] and  Conjecture 2 in~[LY].

\begin{conj}\label{lyconjecture}
Let $S$ be a fake projective plane and $G$ be an ample generator of Pic$(S)$. Then, $H^0(S, G^2\otimes\sigma)=0$ for any torsion line bundle $\sigma$.
\end{conj}

The conjecture above implies a stronger version of Lemma~\ref{vanishinglem}, as follows:
\begin{prop}\label{strongervanishing}
If Conjecture~\ref{lyconjecture} holds for a fake projective plane $S$, then for any ample generator $G$ and torsion $\sigma$ of Pic$(S)$, we have $H^1(S, G^k\otimes\sigma)=0$ for $k\geq 1$ and $H^2(S, G^k\otimes\sigma)=0$ for $k=1,2$, or $k\geq 4$.
\end{prop}
\begin{proof}
If $H^0(S, G^2\otimes\sigma)=0$ for any torsion $\sigma$, then we also have $H^0(S, G\otimes\sigma)=0$ for any torsion $\sigma$. Since $K_S=G^3\otimes\tau$ for some torsion $\tau$, using Serre Duality, we get $$H^2(S, G\otimes\sigma)=H^2(S, G^2\otimes\sigma)=0$$ for any torsion $\sigma$. Now, by Riemann-Roch formula, it follows that $$H^1(S, G\otimes\sigma)=H^1(S, G^2\otimes\sigma)=0$$
for any torsion $\sigma$. 

Finally, $H^1(S, G^3\otimes\sigma)=H^1(S,\tau\otimes\sigma^{-1})=0$ by Serre Duality and Corollary~\ref{h1corollary}. 

(Note that $H^2(S, G^3\otimes\sigma)=H^0(S,\tau\otimes\sigma^{-1})$ is zero if and only if $\sigma\neq\tau$.) 
\end{proof}

With the help of Proposition~\ref{strongervanishing}, we are now able to show that for a fake projective plane $S$ satisfying Conjecture~\ref{lyconjecture}, its pluricanonical line bundle $K^{2p+5}_S$, which is missed in Corollary~\ref{maincoro}, also satisfies the property $N_p$. The essential case is the projective normality (i.e. when $p=0$), which we have singled out in the following theorem. 

\begin{theo}\label{k5thm}
Let $S$ be a fake projective plane satisfying Conjecture~\ref{lyconjecture}. Then $K_S^m$ is projectively normal for $m\geq 4$.
\end{theo}
\begin{proof}
By Corollary~\ref{maincoro}, it remains to prove that $K^5_S$ is projectively normal. Since $K_S=G^3\otimes\tau$, for some torsion $\tau$, using Green's criterion (for the property $N_0$), we just need to show that for any  torsion line bundles $\sigma$, $\delta$, $H^1(S, M_{G^{15}\otimes\sigma}\otimes G^\ell\otimes\delta)=0$ for $\ell= 15s$, where $s\in\mathbb N^+$. As in the proof of Theorem~\ref{vanishingthm2}, after decomposing as $G^{15}\otimes\sigma=G^7\otimes G^8\otimes\sigma$, it suffices to verify that the tensor product maps
$$
	H^0(G^7)\otimes H^0(G^\ell\otimes\delta)\rightarrow
	H^0(G^{7+\ell}\otimes\sigma\otimes\delta),
$$
$$
	H^0(G^8\otimes\sigma)\otimes H^0(G^{7+\ell}\otimes\delta)\rightarrow
	H^0(G^{15+\ell}\otimes\sigma\otimes\delta)
$$
are surjective for $\ell=15s$, $s\in\mathbb N^+$. Using Castelnuovo-Mumford Theorem, it is sufficient to have, respectively,
$$
	H^1(G^{\ell-7}\otimes\delta)=H^2(G^{\ell-14}\otimes\delta)=0
$$
and
$$
	H^1(G^{7+\ell-8}\otimes\sigma\otimes\delta)=H^2(G^{7+\ell-16}
	\otimes\sigma\otimes\delta)=0.
$$

By Proposition~\ref{strongervanishing}, these conditions hold for $\ell=15s$, $s\in\mathbb N^+$.
\end{proof}

\noindent
{\bf Remark.} The above argument implies that $K_S^m$ is projective normal for $m\geq 4$ for fake projective planes with cardinality of automorphism group $|\Aut(S)|>3$,  using the facts, cf. [LY] and the references therein, that $K_S=G^3$ for some ample line bundle $G$ with $H^0(S,G^2)=0$ and letting $\delta$ and $\sigma$ be trivial
in the above argument.

After settling projective normality, we can now prove that the general property $N_p$ holds for $K^{2p+5}_S$. We will see that for $p\geq 1$, the argument does not require Conjecture~\ref{lyconjecture}. The major step is the following vanishing theorem tailored for the purpose. 

\begin{prop}\label{lastvanishingprop}
Let $S$ be a fake projective plane. Let $G$ be an ample generator and $\sigma$, $\delta$ be torsions in Pic$(S)$. For $k\geq 18$ and $r\geq 1$, we have $H^1(S, M^{\otimes r}_{G^k\otimes\sigma}\otimes G^\ell\otimes\delta)=0$ if $\ell\geq 6(r+1)$.
\end{prop}
\begin{proof}
We will prove by induction on $r$. The proof is similar to that of Theorem~\ref{vanishingthm2} and will be brief. 

Suppose $r=1$. We observe that $H^1(M_{G^{k}\otimes\sigma}\otimes G^\ell\otimes\delta)=0$ if $k\geq 18$ and $\ell\geq 12$. To see this, we just need to note that at the beginning of the proof in Theorem~\ref{vanishingthm2}, the roles of $k$ and $\ell$ are symmetric and thus we can exchange $k$ and $\ell$ in the vanishing statement for $r=1$ in Theorem~\ref{vanishingthm2}. The case for $r=1$ is now settled.

Suppose $r=2$.  Since $k\geq 18$, we have a decomposition $G^k\otimes\sigma=G^{k_1}\otimes \cdots\otimes G^{k_p}\otimes\sigma$ for some $p\geq 3$ and $6= k_1\leq\cdots\leq k_p\leq 9$.

To show that $H^1(M^{\otimes r}_{G^k\otimes\sigma}\otimes G^\ell\otimes\delta)=0$, it suffices to verify that $H^1(M^{\otimes{r-1}}_{G^k\otimes\sigma}\otimes G^\ell\otimes\delta)=0$ and the surjectivity for the following mappings:
$$
\varphi_j: H^0(G^{k_j})\otimes H^0(M^{\otimes{r-1}}_{G^k\otimes\sigma}\otimes G^{k_1+\cdots+k_{j-1}+\ell}\otimes\delta)\rightarrow H^0(M^{\otimes{r-1}}_{G^k\otimes\sigma}\otimes G^{k_1+\cdots+k_{j}+\ell}\otimes\delta)
$$ 
for $1\leq j\leq p-1$ and
$$
\varphi_p: H^0(G^{k_p}\otimes\sigma)\otimes H^0(M^{\otimes{r-1}}_{G^k\otimes\sigma}\otimes G^{k_1+\cdots+k_{p-1}+\ell}\otimes\delta)\rightarrow H^0(M^{\otimes{r-1}}_{G^k\otimes\sigma}\otimes G^{k+\ell}\otimes\sigma\otimes\delta).
$$

Now, together with Castelnuovo-Mumford Theorem, we conclude that the cohomology $H^1(S, M^{\otimes r}_{G^k\otimes\sigma}\otimes G^\ell\otimes\delta)$ vanishes if we have
$$
	H^1(M^{\otimes{r-1}}_{G^k\otimes\sigma}\otimes G^\ell\otimes\delta)=0,
$$
$$
	H^1(M^{\otimes{r-1}}_{G^k\otimes\sigma}\otimes G^{\ell-k_1}\otimes\delta)=H^2(M^{\otimes{r-1}}_{G^k\otimes\sigma}\otimes G^{\ell-2k_1}\otimes\delta)=0,
$$
$$	
	H^1(M^{\otimes{r-1}}_{G^k\otimes\sigma}\otimes G^{k_1+\cdots+k_{j-1}+\ell-k_j}\otimes\delta)
	=H^2(M^{\otimes{r-1}}_{G^k\otimes\sigma}\otimes G^{k_1+\cdots+k_{j-1}+\ell-2k_j}\otimes\delta)=0
$$
for $2\leq j\leq p-1$ and
$$	
	H^1(M^{\otimes{r-1}}_{G^k\otimes\sigma}\otimes G^{k+\ell-2k_p}\otimes\sigma\otimes\delta)
	=H^2(M^{\otimes{r-1}}_{G^k\otimes\sigma}\otimes G^{k+\ell-3k_p}\otimes\sigma\otimes\delta)=0.
$$
These all hold if $k\geq 18$ and $\ell-6\geq 6r$, by the induction hypothesis and Theorem~\ref{vanishingthm2} (for the vanishing of $H^2$). The proof is now complete.
\end{proof}

\begin{theo}
Let $S$ be a fake projective plane satisfying Conjecture~\ref{lyconjecture}. Then $K^m_S$ satisfies the property $N_p$ if $m\geq 2p+4$.
\end{theo}
\begin{proof}
By Corollary~\ref{maincoro}, it remains to check that $K^{2p+5}_S$ satisfies the property $N_p$. We have already shown that $K^5_S$ satisfies the property $N_0$ in Theorem~\ref{k5thm} and we now let $p\geq 1$. Let $L_p=K^{2p+5}_S$. According to Green's criterion, we just need to check that $$H^1(S, M^{\otimes r}_{L_p}\otimes L_p^s)=0$$ for $0\leq r \leq p+1$ and $s\geq 1$. 

Let $G$ be an ample generator of Pic$(S)$ and then we can write $$L_p=K^{2p+5}_S=G^{6p+15}\otimes\delta$$ for some torsion line bundle $\delta$. For $p\geq 1$, the desired vanishing statement follows from Proposition~\ref{lastvanishingprop} since $6p+15>18$ and $$(6p+15)s\geq 6p+15>6(p+2)\geq 6(r+1).$$
\end{proof}

\bigskip
\ni {\it Acknowledgement:}  It is a pleasure for the authors to thank Donald Cartwright for providing computer data files for the proof of Theorem 8 which leads to Theorem~\ref{mainthm3} and thank Lawrence Ein and Jun-Muk Hwang for fruitful discussions and providing very useful references. They would also like to thank the anonymous referees for carefully reading the paper and giving helpful suggestions.

\bigskip
\noindent{\bf References} 

\bs 
\ni[BH] P. Banagere, K. Hanumanthu, Syzygies of surfaces of general type. \textit{Geom. Dedicata.} 167 (2013), 123-149.

\ms\ni[BHPV] W. P. Barth, K. Hulek, C. A. M. Peters, A. Van de Ven, \textit{Compact Complex Surfaces}.
 Second edition. Ergebnisse der
Mathematik und ihrer Grenzgebiete. 3. Folge. A Series of Modern
Surveys in Mathematics 4. Springer-Verlag, Berlin, 2004.

\ms\ni [CS] D. Cartwright,  T. Steger,  Enumeration of the $50$ fake projective planes, \textit{C. R. Acad. Sci. Paris, Ser. 1},
348 (2010), 11-13, see also \\
http://www.maths.usyd.edu.au/u/donaldc/fakeprojectiveplanes/

\ms\ni[EL] L. Ein, R. Lazarsfeld, Syzygies and Koszul cohomology of smooth projective varieties of arbitrary dimension. \textit{Invent. Math.} 111 (1993), 51-67.

\ms\ni[G] M. Green, Koszul cohomology and the geometry of projective varieties II. \textit{J. Diff. Geom.} 20 (1984), 279-289.

\ms\ni[GP] F. J. Gallego, B. P. Purnaprajna, Projective normality and syzygies of algebraic surfaces. \textit{J. Reine. Angew. Math.} 506 (1999), 145-180.

\ms\ni[GKMS] S. Galkin, L. Katzarkov, A. Mellit, E. Shinder, Derived categories of Keum's fake projective planes. \textit{Adv. Math.} 278  (2015), 238-253.

\ms\ni[HT] J.-M. Hwang, W.-K. To, Syzygies of compact complex hyperbolic manifolds. \textit{J. Algebraic Geom.} 22 (2013), no. 1, 175-200.

\ms\ni[I] S. P. Inamdar, On syzygies of projective varieties. \textit{Pacific J. Math.} 177 (1997), 71-76

\ms\ni[K] J. Koll\'{a}r, \textit{Shafarevich Maps and Automorphic Forms.} Princeton University Press, Princeton, New Jersey 1995.

\ms\ni[L] R. Lazarsfeld, \textit{Positivity in Algebraic Geometry I}, Springer-Verlag Berlin Heidelberg 2004.

\ms\ni[LPP] R. Lazarsfeld, G. Paresch and M. Popa, Local positivity, multiplier ideals, and syzygies of abelian varieties.  \textit{Algebra Number Theory} 5 (2011), 185-196.

\ms\ni[LY] C.-J. Lai, S.-K. Yeung, Exceptional collection of objects on some fake projective planes. \textit{To appear in Int. Math. Res. Not.}, arXiv:2108.02412  

\ms\ni[P] B. P. Purnaprajna, Some results on surfaces of general type. \textit{Canad. J. Math.} 57 (2005), no. 4, 724-749.

\ms\ni[N] W. Niu, On syzygies of Calabi-Yau varieties and varieties of general type. \textit{Adv. Math.} 343 (2019), 756-788.

\ms\ni [PY1] G. Prasad,  S.-K. Yeung,  Fake projective planes. \textit{Invent. Math.} 168 (2007), 321-370. 

\ms\ni [PY2] G. Prasad,  S.-K. Yeung, Addendum to ``Fake projective planes. Invent. Math. 168, 321-370."
\textit{Invent. Math.} 182, 213-227, 2010.

\ms\ni [R\'em] R\'emy, R., Covolume des groupes $S$-arith\'emiques et faux plans projectifs, [d'apr\`es Mumford, Prasad, Klingler, Yeung, Prasad-Yeung],
S\'eminaire Bourbaki, 10\`eme ann\'ee, 2007-2007, $n^o$ 984.

\ms\ni[Rei] I. Reider, Vector Bundles of Rank 2 and Linear Systems on Algebraic Surfaces. \textit{Ann. Math.}, 127 (1988), 309-316.

\ms\ni[RV] R. Remmert, A. Van de Ven, Zur Funktionentheorie homogener komplexer Mannigfaltigkeiten, \textit{Topology} 2 (1963), 137-157.

\ms\ni[Y] S.-K. Yeung, Classification of fake projective planes. \textit{Handbook of Geometric Analysis (Vol II)}, ALM 13, 389-429.

%\ms\ni[Y2] S.-K. Yeung, Very ampleness of the bicanonical line bundle
%on compact complex 2-ball quotients.  Forum Math. 30(2018), 419-432, erratum, submitted.

\end{document}